\documentclass[12pt]{article}%

\usepackage{color}

\usepackage[utf8]{inputenc}
\usepackage[all]{xy}
\usepackage{amsmath}
\usepackage{amsfonts}
\usepackage{amssymb}
\usepackage{graphicx}
\providecommand{\U}[1]{\protect\rule{.1in}{.1in}}

\newcommand{\N}{{\mathbb N}}
\newcommand{\C}{{\mathbb C}}

\newcommand{\R}{{\mathbb R}}
\newcommand{\Z}{{\mathbb Z}}

\newcommand{\T}{{\mathbb T}}

\newcommand{\Ho}{{\mathbb H}}
\newcommand{\So}{{\mathbb S}}

\newcommand{\Mm}{{\mathcal M}}

\def\bi{\begin{itemize}}
\def\ei{\end{itemize}}
\def\un{\underline}

\newcommand{\ba}{\begin{eqnarray}}
\newcommand{\ea}{\end{eqnarray}}
\newcommand{\bas}{\begin{eqnarray*}}
\newcommand{\eas}{\end{eqnarray*}}
\newcommand{\be}{\begin{equation}}
\newcommand{\ee}{\end{equation}}

\def\dxi{\partial_{x_i}}
\def\dxj{\partial_{x_j}}
\def\dx0{\partial_{x_0}}

\def\dr{\partial_{r}}
\def\duxi{\partial_{\un \xi}}
\def\gxi{\Gamma_{\un \xi}}

\def\am{{\mathcal A}(\So^m)}
\def\cks{CK_{\So^m}}
\def\l2sm{L^2(\So^m, d\sigma_m)\otimes \C_{m+1}}

\newtheorem{theorem}{Theorem}

\newtheorem{lemma}[theorem]{Lemma}

\newtheorem{definition}[theorem]{Definition}
\newenvironment{proof}[1][Proof]{\noindent\textbf{#1.} }{\ \rule{0.5em}{0.5em}}
\newtheorem{preremark}[theorem]{Remark}
\newenvironment{remark}{\begin{preremark}\rm}{\hfill$\Diamond$\end{preremark}}
\newtheorem{prenotation}[theorem]{Notation}

\numberwithin{equation}{section}
\numberwithin{theorem}{section}
\begin{document}

\title{{Clifford Coherent State Transforms on Spheres.}}
\author{Pei Dang\thanks{Faculty of Information Technology, Macau University of Science and Technology.}, \,
Jos\'e  Mour\~ao\thanks{Department of Mathematics and Center for Mathematical Analysis, Geometry and Dynamical Systems, Instituto Superior T\'ecnico, University of  Lisbon.}, \, Jo\~ao P. Nunes\footnotemark[2] \, and Tao Qian\thanks{Department of Mathematics, Faculty of Science and Technology, University of  Macau.}}

\maketitle

\bigskip

\begin{abstract}

We introduce a one-parameter family of transforms, $U^t_{(m)}$,  $t>0$, from the Hilbert space of Clifford algebra 
valued square 
integrable functions on the $m$--dimensional sphere, 
$L^2(\So^{m},d\sigma_{m})\otimes \C_{m+1}$, to the Hilbert spaces,  ${\mathcal M}L^2(\R^{m+1} \setminus \{0\},d\mu_t)$,  of monogenic functions on 
$\R^{m+1}\setminus \{0\}$ which are 
square integrable with respect to  appropriate measures, $d\mu_t$. We prove that these transforms are unitary isomorphisms of the Hilbert spaces
and are extensions of the Segal-Bargman coherent state transform,
 $U_{(1)} : L^2(\So^{1},d\sigma_{1}) \longrightarrow {\mathcal H}L^2({\C \setminus \{0\}},d\mu)$, to 
higher dimensional spheres in the context of Clifford analysis. 
In Clifford analysis it is natural to replace the analytic continuation from $\So^m$ to $\So^m_\C$
as in \cite{Ha1, St, HM}
by the Cauchy--Kowalewski extension from $\So^m$ to $\R^{m+1}\setminus \{0\}$.
One then obtains a unitary isomorphism from an $L^2$--Hilbert space to
an Hilbert space of solutions of the Dirac equation, that is to a Hilbert space of monogenic functions.

\vskip 0.2cm


\end{abstract}

\tableofcontents

\section{Introduction}

In this work, we continue to explore the extensions of coherent state transforms to the context of Clifford 
analysis started in \cite{KMNQ, DG, MNQ, PSS}.
In \cite{MNQ}, an extension of the coherent state transform (CST) to unitary maps
from the spaces of $L^2$ functions  on $M=\R^m$ and on the $m$--dimensional torus, $M=\T^m$,
to the spaces of square integrable monogenic functions on $\R \times M$ was studied.

We consider the cases when $M$ is an $m$--dimensional sphere, $M=S^m$, equipped with the 
$SO(m+1, \R)$--invariant metric of unit volume. These cases are a priori more complicated than
those studied before
as the
transform uses (for $m>1$) the Laplacian and the Dirac operators for the non--flat metrics
on the spheres. We show that there is a unique $SO(m+1, \R)$ invariant measure
on $\R \times \So^m \cong \R^{m+1}\setminus \{0\}$ such that  the natural Clifford CST (CCST)
is unitary. This transform is factorized into a contraction operator given by 
heat operator evolution 
at time $t=1$ followed by Cauchy-Kowalewsky (CK) extension, which 
exactly compensates the contraction for our choice of measure on
$\R^{m+1}\setminus \{0\}$.
In the usual coherent state Segal--Bargmann transforms \cite{Ba, Se, Ha1, Ha2, St, HM}, instead
of the CK extension  to a manifold  with
one more real dimension, one
considers the analytic continuation to a
complexification of the initial manifold (playing the role of phase space of the system).
The CCST is of interest in Quantum Field Theory as it establishes 
natural unitary isomorphisms between  
Hilbert spaces of solutions of the Dirac equation  and  one--particle Hilbert spaces in the Schr\"odinger representation.  The standard CST, on the other hand, 
studies the unitary equivalence of the Schr\"odinger representation
with special K\"ahler representations with the wave functions defined
on the phase space.

In the section \ref{ss-32} we consider a one-parameter family of CCST, using heat operator evolution 
at time $t>0$ followed by CK extension, and we show that, by changing 
the measure on $\R^{m+1}\setminus \{0\}$ to a new Gaussian 
(in the coordinate $\log(|\un x|)$) measure $d\mu_t$, these transforms 
are unitary. As $t$ approaches $0$ (so that the first factor in the 
transform is contracting less than for higher values of $t$)  the measures $d\mu_t$ 
become more concentrated around the radius $|\un x|=1$ sphere
and as  $t \to 0$, the measure $d\mu_t$ converges to
the measure 
$$\delta(y) \, dy \,  d\sigma_m \, ,
$$ 
where $y=\log(|\un x|)$, 
supported on $\So^m$.


\section{Clifford analysis}
\label{ss-ca}

 Let us briefly recall from \cite{BDS, DSS, DS, LMQ, So, FQ, PQS, DQC},  some definitions and results from  Clifford analysis.
 Let $\R_{m+1}$ denote the real Clifford algebra with $(m+1)$ generators, $ e_j, j = 1, \dots, m+1$, identified
with the canonical basis of $\R^{m+1} \subset \R_{m+1}$ and satisfying the relations 
$e_i  e_j +  e_j \ e_i = - 2 \delta_{ij}$. Let $\C_{m+1}=\R_{m+1}\otimes \C$.
We have that $\R_{m+1} = \oplus_{k=1}^{m+1} \R_{m+1}^k$,
where $ \R_{m+1}^k$ denotes the space of $k$-vectors, defined by $\R_{m+1}^0 = \R$ and  $\R_{m+1}^k 
= {\rm span}_\R \{ e_A \, : \, A \subset \{1, \dots , m+1\}, |A| = k\}$, where
$e_{i_1 \dots i_k} = e_{i_1} \dots  e_{i_k}$.  

Notice also that $\R_1 \cong \C$ and $\R_2 \cong \Ho$.
The inner product in $\R_{m+1}$ is defined by
$$
<u, v> = \left(\sum_A u_A  e_A ,  \sum_B v_B  e_B\right) = \sum_A u_A v_A .
$$
The Dirac operator is defined as
$$
 \un D = \sum_{j=1}^{m+1} \,  e_j \, \dxj .
$$
We have that $\un D^2 = - \Delta_{m+1}$.

Consider the subspace of $\R_{m+1}$ of $1$-vectors  
$$
\{\un x = \sum_{j=1}^{m+1} x_j e_j: \,\, x=(x_1,\dots,x_m)\in \R^{m+1}\}\cong \R^{m+1},
$$
which we identify with $\R^{m+1}$. Note that $\un x^2 = - |\un x|^2 = - (x, x).$

Recall that a continuously  differentiable function $f$ on an open domain ${\cal O} \subset \R^{m+1}$,  
with values on $\C_{m+1}$, 
is called (left) monogenic on ${\cal O}$  if it satisfies the Dirac equation (see, for example, \cite{BDS,DSS,So})
\be\nonumber
\un D f(x)  =  \sum_{j=1}^{m+1} \,  e_j \, \dxj \, f(x) = 0.
\ee
For $m=1$, monogenic functions on $\R^2$ correspond to holomorphic
functions of the complex variable $x_1+e_1e_2 \, x_2$.

The Cauchy kernel, 
$$
E(x)  = \frac{\overline{\un{x}}}{|x|^{m+1}}, 
$$
is a monogenic function on $\R^{m+1} \setminus \{0\}$. In the spherical coordinates,
$r = e^y=  |\un x|, \xi= \frac{\un x}{|\un x|}$, 
the Dirac operator reads
\be
\label{e-sdo}
\un D = \frac 1r \, \un \xi \left(r \dr + \Gamma_{\un \xi}\right) = e^{-y} {\un \xi} \left(\partial_y + \Gamma_{\un \xi}\right) ,
\ee
where $\Gamma_{\un \xi}$ is the spherical Dirac operator,
$$
\Gamma_{\un \xi} = - \un \xi \duxi =   - \sum_{i < j} \, e_{ij} \left(x_i  \dxj - x_j \dxi \right).
$$
We see from (\ref{e-sdo}) that the equation for monogenic 
functions in the spherical coordinates is, 
on $\R^{m+1} \setminus \{0\}$, equivalent to
\be
\label{e-mesc}
\un D (f) = 0 \Leftrightarrow \partial_y \- f = -\Gamma_{\un \xi}(f)     \, , \qquad 
r>0.
\ee

The Laplacian $\Delta_x$ has the form
$$
\Delta_x = \dr^2 + \frac mr \dr + \frac{1}{r^2} \Delta_{\un \xi},
$$
where $\Delta_{\un \xi}$ is the Laplacian on the sphere (for the invariant metric).
The relation between the spherical Dirac operator and the spherical Laplace operator
is (see eg \cite{DSS}, (0.16) and section II.1)
\be
\label{e-lop}
\Delta_{\un \xi} = \left((m-1) I -  \gxi\right) \, \gxi 
\ee
Let ${\mathcal H}(m+1, k)$ denote the space of ($\C_{m+1}$--valued) spherical harmonics of degree $k$. These are the eigenspaces 
of the self--adjoint spherical Laplacian, $\Delta_{\un \xi}$,
\ba \label{e-loe}
\nonumber f & \in & {\mathcal H}(m+1, k) \\
\Delta_{\un \xi}  (f) &=& -k(k+m-1) f.
\ea
The spaces  ${\mathcal H}(m+1, k)$ are a direct sum of eigenspaces of the self--adjoint spherical 
Dirac operator
\ba 
\label{e-sdo1} 
\nonumber {\mathcal H}(m+1, k) &=& {\mathcal M}^+(m+1, k) \, \oplus \, {\mathcal M}^-(m+1, k-1)  \\
\gxi (P_k(f)) &=& -k P_k(f) \\
\nonumber \gxi (Q_l (f)) &=& (l+m) Q_l(f) \,   ,  \quad f \in L^2(\So^m, d\sigma_m)\otimes \C_{m+1},
\ea
where $P_k, Q_l$, denote the orthogonal projections on the subspaces ${\mathcal M}^+(m+1, k)$
and ${\mathcal M}^-(m+1, l)$ of $L^2(\So^m, d\sigma_m)\otimes \C_{m+1}$.
The functions in ${\mathcal M}^+(m+1, k)$ and ${\mathcal M}^-(m+1, l)$ 
are in fact the restriction to $\So^m$ of (unique) monogenic functions 
\ba \label{e-ck1}
\nonumber \tilde P_k(f)(\un x) &=& r^k \, P_k(f)\left(\frac{\un x}{|\un x|}\right)  \\
\tilde Q_l(f)(\un x) &=& r^{-(l+m)} \, Q_l(f)\left(\frac{\un x}{|\un x|}\right)  ,  \quad f \in L^2(\So^m, d\sigma_m)\otimes \C_{m+1}, k, l \in \Z_{\geq 0}  ,
\ea
where, for all $f \in L^2(\So^m, d\sigma_m)\otimes \C_{m+1}$, $\tilde P_k(f)$ are monogenic homogeneous polynomials of degree $k$
and $\tilde Q_l(f)$ are monogenic functions on $\R^{m+1}\setminus \{0\}$,
homogeneous of degree $-(l+m)$.

\section{Clifford Coherent State Transforms on Spheres.}
\label{s-3}

\subsection{CCST on spheres and its unitarity}

\begin{definition}
\label{d-am}
Let  $\am$ be the space of analytic $\C_{m+1}$--valued functions on $\So^m$ 
with monogenic continuation to the whole of $\R^{m+1}\setminus \{0\}$.
\end{definition}

\begin{remark}
Let $V$ denote the space of finite linear combinations of spherical monogenics, 
\be
\label{e-ssm}
\nonumber V = {\rm span}_\C \left\{ P_k(f),  Q_l(f), \, \, k, l \in  \Z_{\geq 0} , f \in L^2(\So^{m}, d\sigma_{m}) \otimes \C_{m+1} \right\}.
\ee
We see from (\ref{e-ck1}) that $V \subset \am$. We will denote by $\tilde V$ the space of CK extensions of elements of
$V$ to $\R^{m+1}\setminus \{0\}$ (see (\ref{e-ck1})),
\be
\label{e-ssm2}
\tilde V =  {\rm span}_\C \left\{\tilde  P_k(f),  \tilde Q_l(f), \, \, k, l \in  \Z_{\geq 0} , f \in L^2(\So^{m}, d\sigma_{m}) \otimes \C_{m+1} \right\}.
\ee
\end{remark}

In analogy with the case $m=1$ and also with  the ``usual CST on spheres",
introduced in \cite{St, HM},
we will introduce the CCST
\ba
\label{e-cstm}
\nonumber U_{(m)} \, &:& \,   L^2(\So^{m}, d\sigma_{m}) \otimes \C_{m+1} \longrightarrow 
{\mathcal M}L^2(\R^{m+1}\setminus \{0\}, \tilde \rho_m \,  d^{m+1}x) \\
U_{(m)} &=& CK_{\So^m} \circ  e^{\Delta_{\un \xi}/2} =  e^{-y \Gamma_{\un \xi}} \circ e^{\Delta_{\un \xi}/2} \\
\nonumber U_{(m)}(f)(\un x) &=& \int_{\So^m} \, \tilde K_1(\un x, \un \xi) \, f(\xi) \, 
d \sigma_m  \, , 
\ea
where $CK_{\So^m} \, : \,  \am \longrightarrow \Mm(\R^{m+1}\setminus \{0\})$ denotes the CK extension, 
$K_1$ is the heat kernel on $\So^m$  at time $t=1$
and $\tilde K_1(\cdot , \xi)$ denotes the CK extension
of $K_1$ to  
$\R^{m+1}\setminus \{0\}$ in its first variable (see Lemma \ref{l-1}, (\ref{e-kex}) and (\ref{e-uut}) below).
Our goal is to find (whether there exists) a function $\tilde \rho_m$ on 
$\R^{m+1}\setminus \{0\}$,
$$
\tilde \rho_m(\un x) = \rho_m(y) , \, \qquad y = \log(|\un x|)
$$
which makes the (well defined) map in (\ref{e-cstm}) unitary.
For $m=1$ there is a unique positive answer to the above question 
given by
$$
\rho_1(y) = \frac 1{\sqrt{\pi}} \, e^{-y^2-2y}
$$
so that
$$
\tilde \rho_1(\un x) = \frac 1{\sqrt{\pi}} \, e^{-\log^2(|\un x|)-2\log(|x|)}.
$$

Our main result in the present paper is the following.

\begin{theorem}
\label{t-1}
The map $U_{(m)}$ in (\ref{e-cstm}) is a unitary isomorphism for 
\be
\label{e-rom}
\tilde \rho_m(\un x) =  \frac {e^{-\frac{(m-1)^2}{4}}}{\sqrt{\pi}} \, e^{-\log^2(|\un x|)-2\log(|\un x|)}.
\ee
\end{theorem}

\begin{remark}
\label{r-1}
It is remarkable that the only dependence on $m$ of the corresponding function
$\rho_m(y)$ is in the constant multiplicative factor, $e^{-\frac{(m-1)^2}{4}}$.
\end{remark}

Given the factorized form of $U_{(m)}$ in (\ref{e-cstm}) we have the
diagram
\begin{align}
 \label{d333}
\begin{gathered}
\xymatrix{
&&  \mathcal{M}L^2 (\R^{m+1} \setminus \{0\}, \tilde \rho_m \, d^{m+1}x)  \\
L^2(\So^m, d \sigma_m)\otimes \C_{m+1}  \ar@{^{(}->}[rr]_{e^{{\Delta_{\un \xi}}/2}}  \ar[rru]^{U_{(m)}} && \am 
, \ar[u]_{CK_{\So^m} \, = \, e^{-y \Gamma_{\un \xi}} }
  }
\end{gathered}
\end{align}
We divide the proof of Theorem \ref{t-1} into several lemmas.

\begin{lemma}
\label{l-1}
Let $f \in \am$ and consider its Dirac operator spectral decomposition or, equivalently, its decomposition
in spherical monogenics, 
\be
\label{e-mdec}
f = \sum_{k \geq 0} \, P_k(f) + \sum_{k \geq 0} \, Q_k(f).  
\ee
Then its CK extension is given by
\ba
\label{e-ck2}
\nonumber
\cks(f)(\un x) &=& 
\sum_{k \geq 0} \, \tilde P_k(f)(\un x) + \sum_{k \geq 0} \, \tilde Q_k(f)(\un x)\\
&=&
\sum_{k \geq 0} \, |\un x|^k \, P_k(f)\left(\frac{\un x}{|\un x|}\right)  +  \sum_{k \geq 0} \,  |\un x|^{-(k+m)} \, Q_k(f)\left(\frac{\un x}{|\un x|}\right)  \\
&=& e^{-y \Gamma_{\un \xi}} (f) = |\un x|^{-\Gamma_{\un \xi}} (f)  .
\nonumber  
\ea
\end{lemma}
\begin{proof}
Since $f \in \am$ the two first lines in the right hand side
of (\ref{e-ck2}) are the Laurent expansion of
$\cks(f)(\un x) $ in spherical monogenics (see \cite{DSS}, Theorem 1,
p. 189), uniformly 
convergent on compact subsets of 
$\R^{m+1} \setminus \{0\}$.
The third line in the right hand side follows from
(\ref{e-sdo1}) and 
the fact that $\Gamma_{\un \xi}$ is 
a self-adjoint operator. 
\end{proof}

\begin{remark}
We thus see that, for $f \in \am$, the operator 
of CK extension to $\R^{m+1} \setminus \{0\}$
is
$$
 \cks   =  e^{-y \Gamma_{\un \xi}}  ,
$$
in agreement with (\ref{e-mesc}) and (\ref{e-cstm}).
\end{remark}

\begin{lemma}
\label{l-2}
Let $f \in \l2sm$ and consider its  decomposition
in spherical monogenics, 
\be
\nonumber
f = \sum_{k \geq 0} \, P_k(f) + \sum_{k \geq 0} \, Q_k(f).  
\ee
Then the map
\bas
\label{e-cstm10}
 U_{(m)} \, &:& \,   L^2(\So^{m}, d\sigma_{m}) \otimes \C_{m+1} \longrightarrow 
{\mathcal M}(\R^{m+1}\setminus \{0\})\\
U_{(m)} &=& CK_{\So^m} \circ  e^{\Delta_{\un \xi}/2}= e^{-y \Gamma_{\un \xi}} \circ e^{\Delta_{\un \xi}/2} \,  ,
\eas
where ${\mathcal M}(\Omega)$ denotes the space of monogenic functions
on the open set $\Omega \subset \R^{m+1}$, 
is  well  defined and
\ba
\label{e-ck3}  \nonumber
U_{(m)}(f)(\un x) &=& e^{-y \Gamma_{\un \xi}} \circ e^{\Delta_{\un \xi}/2}(f)(\un x)=\\
\nonumber &=& \sum_{k \geq 0} \, e^{-k(k+m-1)/2} \, |x|^k \, P_k(f)\left(\frac{\un x}
{|x|}\right)  +  \sum_{k \geq 0} \,  
 e^{-(k+1)(k+m)/2} \,  |x|^{-(k+m)}  \,  Q_k(f)\left(\frac{\un x}{|x|}\right)  \\
&=& \int_{\So^m} \tilde K_1(\un x, \un \xi) \, f(\xi) \, d\sigma_m(\xi), 
\ea
where $K_1$ denotes the heat kernel on $\So^m$ at time $t=1$ and $\tilde K_1$ is the CK extension
to $\R^{m+1} \setminus \{0\}$
of $K_1$ in its first variable. 
\end{lemma}
\begin{proof}
From (\ref{e-lop}), (\ref{e-loe}) and (\ref{e-sdo}) we have
\bas   
e^{\Delta_{\un \xi}/2}(f) (\un \eta) &=&
\sum_{k \geq 0} \, e^{-k(k+m-1)/2} \, P_k(f)(\un \eta) +  \sum_{k \geq 0} \,  
 e^{-(k+1)(k+m)/2} \,   Q_k(f)(\un \eta)    \\
&=& \int_{\So^m}  K_1(\un \eta, \un \xi) \, f(\un \xi) \, d\sigma_m(\un \xi) . 
\eas
From \cite{DSS,DQC} we obtain
\be
\label{e-kex}
K_1(\un \eta, \un \xi)  = \sum_{k \geq 0} \, e^{-k(k+m-1)/2} \left(C^+_{m+1, k}(\un \eta, \un \xi) +
C^-_{m+1, k-1}(\un \eta, \un \xi) \right) , 
\ee
where $C^-_{m+1, -1} =0$,
\bas
C^+_{m+1, k}(\un \eta, \un \xi) &=& \frac 1{1-m} \left[-(m+k-1) \, C_k^{(m-1)/2}(<\un \eta, \un \xi>)+
(1-m) \,  C_{k-1}^{(m+1)/2}(<\un \eta, \un \xi>) \,  \un \eta  \wedge \un \xi \right],\\
C^-_{m+1, k-1}(\un \eta, \un \xi) &=& \frac 1{m-1} \left[k \, C_k^{(m-1)/2}(<\un \eta, \un \xi>)+
(1-m) \,  C_{k-1}^{(m+1)/2}(<\un \eta, \un \xi>)  \,  \un \eta  \wedge \un \xi \right], \quad k \geq 1,  \, 
\eas
$\un \eta \wedge \un \xi = \sum_{i<j} (\eta_i\xi_j-\eta_j\xi_i)e_{ij}$
and $C_k^\nu$ denotes the Gegenbauer polynomial of degree $k$ associated with $\nu$. 

Now we prove that $K_1( \cdot , \xi) \in \am$ for every $\xi \in \So^{m+1}$.  
From Lemma \ref{l-1} and (\ref{e-kex}) we conclude that if $K_1( \cdot , \xi) $ has a CK extension then its
Laurent series is given by
\ba
\label{e-uut}
\tilde K_1(\un x, \un \xi)  &=& \tilde K_1^+(\un x, \un \xi) + \tilde K_1^-(\un x, \un \xi)  = \\
\nonumber &=& \sum_{k \geq 0} \, e^{-k(k+m-1)/2} \, |x|^k \, C^+_{m+1, k}\left(\frac{\un x}{|x|}, \un \xi\right) + 
 \sum_{k \geq 1} \, e^{-k(k+m-1)/2} \,  |x|^{-(k+m-1)} \, C^-_{m+1, k-1}\left(\frac{\un x}{|x|}, \un \xi\right)   .
\ea
Let us now show that this series is uniformly convergent in all compact subsets 
of  $\R^{m+1}\setminus \{0\}$. 
From the explicit expressions   for the degree $k$ Gegenbauer polynomials (see e.g. \cite{DSS}, p. 182)
\be
\nonumber
  C_k^{m/2}(<\un \eta, \un \xi>) =
\sum_{j=0}^{[k/2]} \, \frac{(-1)^j 2^{k-2j}(m/2)_{k-j}}{j! (k-2j)!}  \, <\un \eta, \un \xi>^{k-2j}, 
\ee
where $(a)_j = a (a+1) \cdots (a+j-1)$.  We see that
$$
|C_k^{m/2}(<\un \eta, \un \xi>)| \leq \frac{(m+2k)!!}{(m-1)!!}\sum_{j=0}^{[k/2]} \frac{2^{-j}}{j!(k-2j)!}\leq  \frac{(2k+m)!}{(m-1)!}  \,  ,
\quad \forall \eta, \xi \in \So^m . 
$$
Therefore
we obtain that
\ba
\label{e-ineq}
\nonumber |C^+_{m+1, k}(\un \eta, \un \xi)  | & \leq &
\frac{(2k+m-1)!}{(m-1)!} \, (k+m-1) \, + \,  \frac{(2k+m-1)!}{m!} \, m (m-1) \\
&=& \frac{(2k+m-1)!}{(m-1)!} \, (k+2m-2)  \,  ,
\quad \forall \eta, \xi \in \So^m .
\ea
Let $s \in (0,1)$. From the Stirling formula and (\ref{e-ineq}) we conclude that there 
exists $k_0 \in \N$ such that
$$
\nonumber |C^+_{m+1, k}(\un \eta, \un \xi)  |  \leq  e^{s  k(k+m-1)/2} \,  ,
\quad \forall \eta, \xi \in \So^m, \forall k > k_0 ,
$$
and therefore
$$
e^{-k(k+m-1)/2} \, |C^+_{m+1, k}\left(\frac{\un x}{|x|}, \un \xi\right)|  \leq 
e^{-(1-s)  k(k+m-1)/2}, \quad \forall \eta, \xi \in \So^m, \forall k > k_0 .
$$
Then the series,
$$
\tilde K_1^+(\un x, \un \xi) =  \sum_{k \geq 0} \, e^{-k(k+m-1)/2} \, |x|^k \, C^+_{m+1, k}\left(\frac{\un x}{|x|}, \un \xi\right) , 
$$
is uniformly convergent on all compact subsets of $\R^{m+1}$ and therefore its sum is monogenic on  $\R^{m+1}$ in the first variable.
To prove that the second series in  (\ref{e-uut})  is 
 uniformly convergent in compact subsets of
 $\R^{m+1}\setminus \{0\}$ we use 
the fact that the inversion is an isomorphism between $\Mm(\R^{m+1})$ and
$\Mm_0(\R^{m+1} \setminus \{0\})$ 
  (see section 1.6.5 of \cite{DSS})
$$
f \mapsto If,  \, If(\un x) = \frac{\un x}{|x|^{m+1}} \, f\left(\frac{\un x}{|x|^{2}}\right)  \, . 
$$
It is then equivalent  to prove that the 
series
$$
 \left((I\otimes {\rm Id}) (\tilde K_1^-)\right)(\un x, \un \xi)  = \frac{\un x}{|x|^{2}} \, 
 \sum_{k \geq 1} \, e^{-k(k+m-1)/2} \,  |x|^{k} \, C^-_{m+1, k-1}\left(\frac{\un x}{|x|}, \un \xi\right)   ,
$$
is uniformly convergent on compact subsets of $\R^{m+1}$.
But this is a direct consequence of the following inequalities for $|C^-_{m+1, k-1}(\un \eta, \un \xi)|$,
similar to the inequalities
  (\ref{e-ineq})
 for  $|C^+_{m+1, k}(\un \eta, \un \xi)|$,
\be
\label{e-ineq2}
|C^-_{m+1, k-1}(\un \eta, \un \xi)  |  \leq \frac{(2k+m-1)!}{(m-1)!} \, (k+m-1)  \,  ,
\quad \forall \eta, \xi \in \So^m \, .
\ee
We have thus established that $\tilde K_1( \cdot , \xi) \in \Mm(\R^{m+1} \setminus \{0 \}),   \, \forall \xi \in \So^m$
with Laurent series given by (\ref{e-uut}).  Analogously we can show that
$\tilde K_1( \cdot , \cdot) \in C^\infty(\R^{m+1} \setminus \{0 \}  \times \So^m) \otimes \C_{m+1}$.  

 From (\ref{e-ineq}) and  (\ref{e-ineq2}), we also obtain,
\bas
|P_k(f)(\eta)| &=& \left|\int_{\So^m} \, C^+_{m+1, k}(\un \eta, \un \xi) \, f(\xi) \, d\sigma_m\right|
\leq \frac{(2k+m-1)!}{(m-1)!} \, (k+2m-2)  \, ||f||,  \,  \\
|Q_{k-1}(f)(\eta)| &=&  \left|\int_{\So^m} \, C^-_{m+1, k-1}(\un \eta, \un \xi) \, f(\xi) \, d\sigma_m\right|
\leq \frac{(2k+m-1)!}{(m-1)!} \, (k+m-1)  \, ||f|| ,  \,   \quad \forall \eta \in \So^m \, .
\eas
As in the case of $\tilde K_1(\cdot , \xi)$,
these inequalities imply that, for every $f \in L^2(\So^m, d\sigma_m) \otimes \C_{m+1}$, 
the Laurent series for $U_m(f)$ in (\ref{e-ck3})
is uniformly convergent on compact subsets of $\R^{m+1}\setminus \{0\}$.

\end{proof}

\begin{lemma}
\label{l-3}
The map $U_{(m)}$ in (\ref{e-cstm}) and (\ref{e-ck3}) is an isometry for the measure factor
$\tilde \rho_m$ given by (\ref{e-rom}).
\end{lemma}
\begin{proof}
Given the $SO(m+1, \R)$--invariance of the measures on $\So^m$ and on $\R^{m+1}\setminus \{0\}$ in (\ref{e-cstm}) and (\ref{e-rom}), so that (\ref{e-mdec}) is an orthogonal decomposition and so is (\ref{e-ck3}), 
we see that to prove isometricity of $U_{(m)}$ it is sufficient to prove 
\ba
\label{e-iso1}
\nonumber ||U_{(m)}(P_k(f))|| &=& ||P_k(f)||,   \\
||U_{(m)}(Q_k(f))|| &=& ||Q_k(f)||, 
\ea
for all $k \in \Z_{\geq 0}$ and $f \in \l2sm$.
We have
\bas
||U_{(m)}(P_k(f))||^2 &=& e^{-k(k+m-1)} \, \int_o^\infty \, r^{2k} \rho_m(\log(r)) r^m dr \,  ||P_k(f)||^2 , \\
||U_{(m)}(Q_{k-1}(f))||^2 &=& e^{-k(k+m-1)} \, \int_o^\infty \, r^{-2(k-1+m)} \rho_m(\log(r)) r^m dr \,  ||Q_{k-1}(f)||^2  \, , 
\eas
and therefore isometricity is equivalent to the following two  infinite systems  of equations setting constraints
on the Laplace transform of the function $\rho_m(y)$. The system coming from the $P_k$ is
\be
\label{e-const1}
\int_\R \, \rho_m(y) \, e^{y(2k+m+1)} \, dy =  e^{k(k+m-1)} \, , \qquad  k \in \Z_{\geq 0} , 
\ee
and the system coming from the $Q_k$ is
\be
\label{e-const2}
\int_\R \, \rho_m(y) \, e^{-y(2k+m-3)} \, dy =  e^{k(k+m-1)} \, ,  \qquad  k \in \Z_{\geq 0} .
\ee
It is easy to verify that   the function $\rho_m$ corresponding to $\tilde \rho_m$ 
in (\ref{e-rom})
$$
\rho_m(y) =   \frac {e^{-\frac{(m-1)^2}{4}}}{\sqrt{\pi}} \, e^{-y^2-2y}.
$$
satisfies both (\ref{e-const1}) and  (\ref{e-const2}).
\end{proof}

\begin{remark}
\label{r-b1}
Notice that each of the two systems  (\ref{e-const1}) and 
 (\ref{e-const2}) 
determines 
$\rho_m$ uniquely so that it is remarkable that they both give the same solution.
\end{remark}

\begin{proof} (of Theorem \ref{t-1}). From Lemmas \ref{l-1}, \ref{l-2} and \ref{l-3} we
see that the only missing part is the surjectivity of $U_{(m)}$. But this follows from the fact that 
the space $\tilde V$ in  (\ref{e-ssm2})
is dense, with respect to uniform convergence
on compact subsets, in the space of monogenic functions on $\R^{m+1}\setminus \{0\}$ and therefore is  also
dense 
on  ${\mathcal M}L^2(\R^{m+1}\setminus \{0\}, \tilde \rho_m \,  d^{m+1}x)$ since this has finite measure. Since 
the image of an isometric map is closed and the
image of $U_{(m)}$ contains $\tilde V$ we conclude that
$U_{(m)}$ is surjective. 
\end{proof}

As we mentioned in the introduction the mechanism for the unitarity of the
CST, $U_{(m)}$, was its
factorization   into a contraction given by 
heat operator evolution 
at time $t=1$ followed by Cauchy-Kowalewsky (CK) extension, 
which 
exactly compensates the contraction, given our choice of measure on
$\R^{m+1}\setminus \{0\}$.

\subsection{One-parameter family of unitary transforms}
\label{ss-32}

In the present section we will  consider a one-parameter family of transforms, using heat operator evolution 
at time $t>0$ followed by CK extension. We show that, by changing 
the measure on $\R^{m+1}\setminus \{0\}$ to a new Gaussian 
(in the coordinate $\log(|x|)$) measure $$d\mu_t =
\tilde \rho^t_m \,  d^{m+1}x,$$ these transforms 
are unitary.
Thus we consider the transforms
\ba
\label{e-cstmb}
\nonumber U^t_{(m)} \, &:& \,   L^2(\So^{m}, d\sigma_{m}) \otimes \C_{m+1} \longrightarrow 
{\mathcal M}L^2(\R^{m+1}\setminus \{0\}, \tilde \rho^t_m \,  d^{m+1}x) \\
U^t_{(m)} &=& CK_{\So^m} \circ e^{t\Delta_{\un \xi}/2}
= e^{-y\Gamma_{\un \xi}} \circ e^{t\Delta_{\un \xi}/2} \\
\nonumber U^t_{(m)}(f)(\un x) &=& \int_{\So^m} \, \tilde K_t(\un x, \un \xi) \, f(\xi) \, 
d \sigma_m  \, , 
\ea
where
$\tilde K_t(\cdot , \xi)$ denotes the CK extension
of $K_t$ to  
$\R^{m+1}\setminus \{0\}$ in its first variable.

Our goal is to find (whether there exist), for every $t>0$, a function $\tilde \rho^t_m$ on 
$\R^{m+1}\setminus \{0\}$,
$$
\tilde \rho^t_m(\un x) = \rho^t_m(y) , \, 
$$
which makes the (well defined) map in (\ref{e-cstmb}) unitary.
Again, for $m=1$, there is a unique positive answer to the above question 
given by
$$
\rho_1^t(y) = \frac 1{\sqrt{t\pi}} \, e^{- \frac{y^2}t -2y}
$$
so that
$$
\tilde \rho^t_1(\un x) = \frac 1{\sqrt{t\pi}} \,  e^{- \frac 1t \log^2(|x|)-2\log(|x|)}.
$$

We then have
\begin{theorem}
\label{t-1b}
The map $U^t_{(m)}$ in (\ref{e-cstmb}) is a unitary isomorphism for 
\be
\label{e-romb}
\tilde \rho^t_m(\un x) =  \frac {e^{-\frac{t(m-1)^2}{4}}}{\sqrt{t\pi}} \, e^{- \frac 1t \log^2(|x|)-2\log(|x|)}.
\ee
\end{theorem}

Given the factorized form of $U^t_{(m)}$ in (\ref{e-cstmb}) we have the
diagram
\begin{align}
 \label{d333b}
\begin{gathered}
\xymatrix{
&&  \mathcal{M}L^2 (\R^{m+1} \setminus \{0\}, \tilde \rho^t_m \, d^{m+1}x)  \\
L^2(\So^m, d \sigma_m)\otimes \C_{m+1}  \ar@{^{(}->}[rr]_{e^{{t\Delta_{\un \xi}}/2}}  \ar[rru]^{U^t_{(m)}} && \am 
, \ar[u]_{CK_{\So^m} \, = \, e^{-y \Gamma_{\un \xi}}}
  }
\end{gathered}
\end{align}
Again we divide the proof of Theorem \ref{t-1b} into several lemmas.
Notice, however, that Lemma \ref{l-1} remains unchanged. 

%

\begin{lemma}
\label{l-2b}
Let $f \in \l2sm$ and consider its  decomposition
in spherical monogenics, 
\be
\nonumber
f = \sum_{k \geq 0} \, P_k(f) + \sum_{k \geq 0} \, Q_k(f).  
\ee
Then the map
\bas
\label{e-cstm10t}
 U^t_{(m)} \, &:& \,   L^2(\So^{m}, d\sigma_{m}) \otimes \C_{m+1} \longrightarrow 
{\mathcal M}(\R^{m+1}\setminus \{0\})\\
U^t_{(m)} &=& CK_{\So^m} \circ  e^{t\Delta_{\un \xi}/2}= e^{-y \Gamma_{\un \xi}} \circ e^{t\Delta_{\un \xi}/2} \,  ,
\eas 
is  well  defined and
\ba
\label{e-ck3b}  \nonumber
U^t_{(m)}(f)(\un x) &=& e^{-y \Gamma_{\un \xi}} \circ e^{t\Delta_{\un \xi}/2}(f)(\un x)=\\
\nonumber &=& \sum_{k \geq 0} \, e^{-tk(k+m-1)/2} \, |x|^k \, P_k(f)\left(\frac{\un x}
{|x|}\right)  +  \sum_{k \geq 0} \,  
 e^{-t(k+1)(k+m)/2} \,  |x|^{-(k+m)}  \,  Q_k(f)\left(\frac{\un x}{|x|}\right)  \\
&=& \int_{\So^m} \tilde K_t(\un x, \un \xi) \, f(\xi) \, d\sigma_m(\xi), 
\ea
where  $\tilde K_t$ is the CK extension
to $\R^{m+1} \setminus \{0\}$
of $K_t$ in its first variable. 
\end{lemma}

\begin{proof} The proof is identical to the proof of Lemma \ref{l-2}. 
The Gaussian form (in $k$) of the coefficients 
coming from $e^{t\Delta_{\un \xi}/2}$
and the inequalities
(\ref{e-ineq}), (\ref{e-ineq2}) again  imply  that
$\tilde K_t(\cdot , \xi)$ and $U^t_{(m)}(f)$ are monogenic on
$\R^{m+1}\setminus \{0\}$ and their 
 Laurent 
series are given by
\ba
\label{e-uutb}
\tilde K_t(\un x, \un \xi)  &=& \tilde K_t^+(\un x, \un \xi) + \tilde K_t^-(\un x, \un \xi)  \\
\nonumber &=& \sum_{k \geq 0} \, e^{-tk(k+m-1)/2} \, |x|^k \, C^+_{m+1, k}\left(\frac{\un x}{|x|}, \un \xi\right) + 
 \sum_{k \geq 1} \, e^{-tk(k+m-1)/2} \,  |x|^{-(k+m-1)} \, C^-_{m+1, k-1}\left(\frac{\un x}{|x|}, \un \xi\right)   ,
\ea
and by (\ref{e-ck3b}).
\end{proof}

\begin{lemma}
\label{l-3b}
The map $U^t_{(m)}$ in (\ref{e-cstmb}) and (\ref{e-ck3b}) is an isometry for the measure factor
$\tilde \rho^t_m$ given by (\ref{e-romb}).
\end{lemma}
\begin{proof}
Given the $SO(m+1, \R)$--invariance of the measures on $\So^m$ and on $\R^{m+1}\setminus \{0\}$ in (\ref{e-cstmb}) and (\ref{e-romb})
we see that to prove isometricity of $U^t_{(m)}$ it is sufficient to prove 
\ba
\label{e-iso1b}
\nonumber ||U^t_{(m)}(P_k(f))|| &=& ||P_k(f)||,  \\
||U^t_{(m)}(Q_k(f))|| &=& ||Q_k(f)||,
\ea
for all $k \in \Z_{\geq 0}$ and $f \in \l2sm$.
Again, isometricity is equivalent to the following two  infinite systems  of equations setting constraints
on the Laplace transform of the functions $\rho^t_m(y)$. The system coming from the $P_k$ is
\be
\label{e-const1b}
\int_\R \, \rho^t_m(y) \, e^{y(2k+m+1)} \, dy =  e^{tk(k+m-1)} \, , \qquad  k \in \Z_{\geq 0} , 
\ee
and the system coming from the $Q_k$ is
\be
\label{e-const2b}
\int_\R \, \rho^t_m(y) \, e^{-y(2k+m-3)} \, dy =  e^{tk(k+m-1)} \, ,  \qquad  k \in \Z_{\geq 0} .
\ee
It is easy to verify that   the function $\rho^t_m$ corresponding to $\tilde \rho^t_m$ in (\ref{e-romb}) 
satisfies both (\ref{e-const1b}) and  (\ref{e-const2b}).
\end{proof}

\begin{proof} 
The proof of Theorem \ref{t-1b} is completed exactly 
as the proof of Theorem \ref{t-1} so that we omit it here.
\end{proof}

\newpage

{\bf \large{Acknowledgements:}} 
The authors were partially
supported by Macau Government FDCT through the project 099/2014/A2,
{\it Two related topics in Clifford analysis}, by the Macao Science and Technology 
Development Fund, MSAR, Ref. 045/2015/A2
and by the University of Macau Research Grant  MYRG115(Y1-L4)-FST13-QT. 
 The authors  
JM and JPN were also partly supported by
FCT/Portugal through the projects UID/MAT/04459/2013 and PTDC/MAT-GEO/3319/2014.

\end{document}